\documentclass[12pt]{article}%
\usepackage{amsfonts}%
\usepackage{amsmath}%
\usepackage{amssymb}%
\usepackage{graphicx}

\providecommand{\U}[1]{\protect\rule{.1in}{.1in}}
\newtheorem{theorem}{Theorem}

\newtheorem{lemma}[theorem]{Lemma}

\newtheorem{proposition}[theorem]{Proposition}

\newenvironment{proof}[1][Proof]{\noindent\textbf{#1.} }{\ \rule{0.5em}{0.5em}}
\begin{document}

\title{An inequality of Harish-Chandra}
\author{Nolan R. Wallach}
\maketitle

\begin{abstract}
In paper I of his masterpiece Harmonic Analysis on Real Reductive Groups,
Harish-Chandra included an important inequality that is useful in proving that
certain key integrals depending on a parameter converge for large values of
the parameter. His proof involved the Tarski-Seidenberg Theorem. The purpose
of this note is an elementary proof of the inequality which is an expansion of
the idea in my Real Reductive Groups I. This exposition fixes several critical
misprints in the original and can be considered to be an erratum for the book. \ 

\end{abstract}

\section{Introduction}

In \cite{Plancherel1} Lemma 32.2 Harish-Chandra stated an important inequality
with a proof that asserted that it was a consequence of the Tarski-Seidenberg
Theorem as stated in \cite{Hor1} on page 276. In the next section, after
setting up the appropriate notation, I will explain the inequality in
question. In my book \cite{RRGI} I stated this inequality as Lemma 4.A.2.3
(with an unfortunate misprint) and gave two "proofs" the first saying at a
certain point it follows from the Tarski-Seidenberg Theorem and the second
involved an elementary induction. Unfortunately the second proof also had some
misprints and missing symbols. This note can therefore be considered to be an
erratum to my book. Also, in the next section the necessary algebra will be
done in order to apply the elementary proof. No one was more careful in his
proofs than Harish-Chandra so I have no doubt that he was certain that he was
sketching a complete proof but the reason he gives in \cite{Plancherel1} (top
of page 193) seems to need further explanation. Since I use this inequality in
\cite{WPT} and since my exposition of it in \cite{RRGI} is riddled with typos
I feel that it is worthwhile to distribute this note which hopefully clears up
the subject. Also the proof given leaves less to the reader than a properly
edited version of the original. Included, as an appendix, is a compendium of
the results on the structure of parabolic subgroups of real reductive groups
that we use in this note.

\section{The set-up}

The notation for the material in this section can be found in the appendix on
parabolic subgroups. As in the appendix, $G$ will denote a real reductive
group as \ in \cite{BoWa} and $K$ a maximal compact subgroup of $G$ with
corresponding Cartan involution $\theta$. All of the terminology on parabolic
subgroups undefined in this section can be found in the appendix. Let $P_{o}$
be a minimal parabolic subgroup of $G$ with Langlands decomposition
$P_{o}=\left(  M_{o}\cap K\right)  A_{o}N_{o}$. We assume that $\mathfrak{a}%
_{o}\left(  =Lie(A_{o})\right)  $ is contained in $\mathfrak{q}=\{X\in
\mathfrak{g}|\theta X=-X\}$. Let $P$ be a parabolic subgroup of $G$ containing
$P_{o}$ with Langlands decomposition $P={}^{o}M_{P}A_{P}N_{P}$. Let $\bar
{P}=\theta(P)$ then $\bar{P}$ is a parabolic subgroup of $G$ containing
$\bar{P}_{o}=\theta(P_{o})$ the Langlands decomposition of $\bar{P}$ is
$\bar{P}={}^{o}M_{P}A_{P}\bar{N}_{P}$ with $\bar{N}_{P}=\theta(N_{P}) $. Let
\[
a_{\bar{P}}:G\rightarrow A_{P}
\]
be as in the appendix (note that ${}^{o}M_{\bar{P}}={}^{o}M_{P}$ and
$A_{\bar{P}}=A_{P}$).

Define $\rho_{P}\in\mathfrak{a}_{P}^{\ast}$ by
\[
\rho_{P}(h)=\frac{1}{2}\mathrm{tr}(ad(h)_{|\mathfrak{n}_{P}}).
\]
Let $\log:A_{P}\rightarrow\mathfrak{a}_{P}$ denote the inverse to the map
$\exp:\mathfrak{a}_{P}\rightarrow A_{P}$. If $\nu\in\left(  \mathfrak{a}%
_{P}\right)  _{\mathbb{C}}^{\ast}$ then define $a^{\nu}=e^{\nu(\log a)}$.

Let $\left\langle ...,...\right\rangle $ be as in the appendix and let
$\left\Vert ...\right\Vert $ be the corresponding norm on $\mathfrak{g}$. The
Harish-Chandra inequality that is the subject of this note is (the logarithm of)

\begin{proposition}
\label{H-C-Inequality}There exist $C>0$ and $m>0$ such that if $X\in
\mathfrak{n}_{P}$ then%
\[
a_{\bar{P}}(\exp(X))^{\rho_{P}}\geq C(1+\left\Vert X\right\Vert ^{2})^{m}.
\]
Notice that $\exp:\mathfrak{n}_{P}\rightarrow N_{P}$ is a diffeomrophism. Thus
if $\log:N_{P}\rightarrow\mathfrak{n}_{P}$ is the invers map then the
inequality reads%
\[
a_{\bar{P}}(n)^{\rho_{P}}\geq C(1+\left\Vert \log n\right\Vert ^{2})^{m},n\in
N_{P}.
\]

\end{proposition}

For the rest of this section we fix $P_{o}$ and $P$ as above with their given
Langlands decompositions. Let $H\in\mathfrak{a}_{P}$ be such that $adH $
defines a grade, $\mathfrak{g=\oplus}_{j\in\mathbb{Z}}\mathfrak{g}_{j}$ (that
is, $\mathfrak{g}_{j}$ is the $j$ eigenspace for $adH$) such that
$\mathfrak{p=\oplus}_{j\geq0}\mathfrak{g}_{j},\mathfrak{m}_{P}=\mathfrak{g}%
_{0}$ and $\mathfrak{n}_{P}=\mathfrak{\oplus}_{j\geq1}\mathfrak{g}_{j}$. Also,
note that the decomposition $\mathfrak{g=\oplus}_{j\in\mathbb{Z}}%
\mathfrak{g}_{j}$ is orthogonal. We will write $\left(  \mathfrak{n}%
_{P}\right)  _{j}=\mathfrak{g}_{j}$ if $j>0$. As we have observed in the
appendix $\left(  \mathfrak{n}_{P}\right)  _{1}\neq0$. Let $q$ be the maximum
of the $j$ such that $\left(  \mathfrak{n}_{P}\right)  _{j}\neq0$. Let
$r=\dim\mathfrak{n}_{p}$ and let $V=\wedge^{r}\mathfrak{g}$ and extend the
inner product, $\left\langle ...,...\right\rangle $, to $V$ in the usual way.
Let $\xi$ be a unit vector in $\wedge^{r}\mathfrak{n}_{\bar{P}}\mathfrak{.}$
Let $\sigma(g)=\wedge^{r}ad(g)$ we will also write $\sigma(X)=d\sigma(X)$. Set%
\[
Z=\sigma(U(\mathfrak{g}))\xi=\sigma(U(\mathfrak{n}_{P}\mathfrak{))\xi.}
\]
Set $Z_{j}$ equal to the $\sigma(H)$ eigenspace for eigenvalue $-2\rho
_{P}(H)+j$ for $j=0,1,...$. Then the decomposition
\[
Z=\oplus_{j\geq0}Z_{j}
\]
is orthogonal.

Observing that if $\bar{n}\in N_{\bar{P}},m\in{}^{o}M_{P}={}^{o}M_{\bar{P}}$
and $a\in A_{o}$ then $\sigma(ma\bar{n})\xi=\det(ad(m)_{|\mathfrak{\bar{n}%
}_{P}})a^{-2\rho_{P}}\xi$. Also, if $m\in{}^{o}M_{P}$ then $\left\vert
\det(ad(m)_{|\mathfrak{\mathfrak{n}}_{\bar{P}}})\right\vert =1$. Thus, if
$g=\bar{n}mak$ with $\bar{n}\in N_{\bar{P}}$,$a\in A_{P},m\in{}^{o}M_{P}$ then%
\[
\left\Vert \sigma(g)^{-1}\xi\right\Vert =a^{2\rho_{P}}=a_{\bar{P}}%
(g)^{2\rho_{P}}.
\]
We note that $\left\Vert \sigma(\exp-X)\xi\right\Vert ^{2}$ is a polynomial
in$X$ that is strictly positive. The Tarski-Seidenberg Theorem implies (see
\cite{Hor2}, Example A.2.7 p. 368) that if $f(x)$ is a polynomial function on
$\mathbb{R}^{n}$ such that $f(x)>0$ for all $x\in\mathbb{R}^{n}$ then there
exists $C>0$ and $r$ such that%
\[
f(x)\geq C(1+\left\Vert X\right\Vert ^{2})^{r}.
\]
Which seems to prove the above proposition. However, H\"{o}rmander shows on
page $369$ that there are polynomial functions, $f$, on $\mathbb{R}^{n}$ such
that $f(x)>0$ for all $x\in\mathbb{R}^{n}$ but any inequality of the type
displayed above has $r<0$. This is, however not Harish-Chandra's reference to
the Tarski-Seidenberg Theorem it is in, fact, to an earlier exposition of
H\"{o}rmander in \cite{Hor1} p. 276. The only pertinent result on that page is
Lemma 2.1 which doesn't directly prove \ the proposition since it gives
asymptotic behavior and not lower bounds. Obviously, Harish-Chandra felt that
the transition to the appropriate lower bound is obvious. The next lemma
reduces the proposition above to the Lemma that is the subject of the next
section. We will be using the notation $X_{i}$ for the orthogonal projection
of $X\in\mathfrak{n}_{P}$ to $\left(  \mathfrak{n}_{P}\right)  _{j}$ and
$z_{j}$ of the orthogonal projection of $z\in Z$ to $Z_{j}$.

In light of lemma \ref{MyInequality} in the next section the following result,
taking (in the notation of lemma \ref{MyInequality} and this lemme)
$\varphi=\psi,$ $W=\mathfrak{n}_{P},T_{j}w=\sigma(-w)\xi$, implies Proposition
\ref{H-C-Inequality}

\begin{lemma}
\label{Setup}Set $\psi(X)=\sigma(\exp(-X))\xi$ for $X\in\mathfrak{n}_{P} $
then $\psi(X)_{0}=\xi$, $\psi(X)_{1}=-\sigma(X_{1})\xi$ and if $2\leq j\leq q$
then%
\[
\psi(X)_{j}=-\sigma(X_{j})\xi+u_{j}(X_{1}+X_{2}+...+X_{j-1})
\]
with $u_{j}$ a polynomial on $\mathfrak{n}_{1}\oplus...\oplus\mathfrak{n}%
_{j-1}$ with values in $V_{j}$ of degree $j$. Also the map $\mathfrak{n}%
_{j}\rightarrow V_{j}$, $X\mapsto\sigma(X)\xi$ is injective for $1\leq j\leq
q$.
\end{lemma}

\begin{proof}
Note that
\[
\psi(X)=\xi+\sum_{j\geq1}\frac{(-1)^{j}}{k!}\sigma(X)^{j}\xi
\]
since every every component in $X\in\mathfrak{n}$ is at least of degree $1$ in
the grade it is clear that%
\[
\psi(X)_{0}=\xi.
\]
Similarly, if $k>j\geq0$ then $\left(  \sigma(X)^{k}\xi\right)  _{j}=0$. Thus
$\psi(X)_{1}=-\sigma(X_{1})\xi$ and if $j>1$ then%
\[
\psi(X)_{j}=\sum_{k=1}^{j}\frac{(-1)^{k}}{k!}\sigma(X)^{k}\xi.
\]
using the above observations, if $j\leq q$ then
\[
\psi(X)_{j}=\sum_{k=1}^{j}\frac{(-1)^{k}}{k!}\left(  \sigma(X_{1}%
+...+X_{j})^{k}\xi\right)  _{j}.
\]
Hence%
\[
\psi(X)_{j}=-\sigma(X_{j})\xi+\sum_{k=2}^{j}\frac{(-1)^{k}}{k!}\left(
\sigma(X_{1}+...+X_{j-1})^{k}\xi\right)  _{j}
\]
since
\[
\left(  \sigma(X)^{k-l}\sigma(X_{j})\sigma(X)^{l-1}\xi\right)  _{j}=0
\]
if $k-l\geq0$, $l-1\geq0$ and one of $k-l>0$ or $l-1>0$. Take
\[
u_{j}(X_{1},...,X_{j-1})=\sum_{k=2}^{j}\frac{(-1)^{k}}{k!}\left(  \sigma
(X_{1}+...+X_{j-1})^{k}\xi\right)  _{j}\text{.}
\]
Finally,if $X\in\mathfrak{n}_{P}$ and $\sigma(X)\xi=0$ then
$ad(X)\mathfrak{\bar{n}}_{P}\mathfrak{\subset\mathfrak{\bar{n}}}_{P}$. But if
$X\neq0$ then $[X,\theta X]\in\mathfrak{a}-\{0\}$.
\end{proof}

\section{A key lemma\label{key}}

\begin{lemma}
\label{MyInequality}Let $\left(  W,\left\langle ...,...\right\rangle \right)
$ and $(V,(...,...))$ be non-zero finite dimensional Hilbert spaces over
$\mathbb{R}$ such that $W=\oplus_{j=1}^{q}W_{j}$ and $V=\oplus_{_{j=0}}%
^{p}V_{j}$ orthogonal direct sums. If $w\in W$ (resp. $v\in V)$ set $w_{j}$
(resp. $v_{j}$) equal to its orthogonal projection into $W_{j}$ (resp. $V_{j}%
$). Let $\varphi:W\rightarrow V$ be a map satisfying:

If $j=0$ then $\varphi(w)_{0}$ is non-zero and constant.

If $j=1$ then%
\[
\varphi(w)_{1}=T_{1}w_{1}\text{.}
\]

If $1<j\leq q$ then%
\[
\varphi(w)_{j}=T_{j}w_{j}+u_{j}(w_{1}+...+w_{j-1})
\]
with $T_{j}$ an injective linear map of $W_{j}$ into $V_{j}$ for $j=1,...,q$
and $u_{j}$ is a polynomial of degree $d_{j}$ with values in $V_{j}.$

Then there exists $m>0$ and $C>0$ such that
\[
\left\Vert \varphi(w)\right\Vert \geq C(1+\left\Vert w\right\Vert ^{2})^{m}.
\]

\end{lemma}

\begin{proof}
Let $v_{0}=\varphi(w)_{0}$ (for all $w\in W$). Replacing $\varphi$ with
$\frac{1}{\left\Vert v_{0}\right\Vert }\varphi$ we may assume that $\left\Vert
v_{o}\right\Vert =1$. Define%
\[
\varphi_{r}(w)=\sum_{j=0}^{r}\varphi(w)_{j}.
\]
Then the conditions on $\varphi$ imply that if $r\leq q$ then $\varphi
_{r}(w)=\varphi_{r}(w_{1}+...+w_{r})$. We first prove by induction on $r$ that
there exist $C_{r}>0$ and $1\geq m_{r}>0$ such that
\[
\left\Vert \varphi_{r}(w)\right\Vert ^{2}\geq C_{r}(1+\left\Vert
w_{1}+...+w_{r}\right\Vert ^{2})^{m_{r}}.
\]
If $r=1$ then
\[
\varphi_{1}(w_{1})_{0}+\varphi_{1}(w_{1})_{1}=v_{0}+T_{1}w_{1}.
\]
so%
\[
\left\Vert \varphi_{1}(w_{1})\right\Vert ^{2}=1+\left\Vert T_{1}%
w_{1}\right\Vert ^{2}
\]
since $T_{j}$ is injective \thinspace$\left\Vert T_{j}y_{j}\right\Vert
\geq\lambda_{j}\left\Vert y_{j}\right\Vert $. Let $c_{1}$ be the minimum of
$1$ and $\lambda_{1}^{2}$. Then
\[
\left\Vert \varphi_{1}(w_{1})\right\Vert ^{2}\geq c_{1}(1+\left\Vert
w_{1}\right\Vert ^{2}).
\]
Assume the result for $r-1\geq1$. If $r>q$ then since $r-1\geq q$ we have
\[
\left\Vert \varphi_{r}(w)\right\Vert ^{2}=\left\Vert \varphi_{r-1}%
(w)+\varphi(w)_{r}\right\Vert ^{2}\geq\left\Vert \varphi_{r-1}(w)\right\Vert
^{2}\geq C_{r-1}(1+\left\Vert w_{1}+...+w_{q}\right\Vert ^{2})^{m_{r-1}}
\]%
\[
=C_{r-1}(1+\left\Vert w\right\Vert ^{2})^{m_{r-1}}.
\]
So we may assume that $r\leq q$. Proceeding with the inductive step
\[
\left\Vert \varphi_{r}(w_{1}+...+w_{r})\right\Vert ^{2}=\left\Vert
\varphi_{r-1}(w_{1}+...+w_{r-1})\right\Vert ^{2}+\left\Vert T_{r}w_{r}%
+u_{r}(w_{1}+...+w_{r-1}\right\Vert ^{2}.
\]
Let
\[
S_{r}=\{w_{1}+...+w_{r}|\left\Vert T_{r}w_{r}\right\Vert >2\left\Vert
u_{r}(w_{1},...,w_{r-1})\right\Vert \}.
\]
If
\[
w_{1}+...+w_{r}\in S_{r}
\]
then%
\[
\left\Vert T_{r}w_{r}+u_{r}(w_{1}+...+w_{r-1})\right\Vert \geq\frac{1}%
{2}\left\Vert T_{r}w_{r}\right\Vert \geq\frac{\lambda_{r}}{2}\left\Vert
w_{r}\right\Vert .
\]
Thus%
\[
\left\Vert \varphi_{r}(w_{1}+...+w_{r})\right\Vert ^{2}\geq\left\Vert
\varphi_{r-1}(w_{1}+...+w_{r-1})\right\Vert ^{2}+\frac{\lambda_{r}^{2}}%
{4}\left\Vert w_{r}\right\Vert ^{2}
\]%
\[
\geq C_{r-1}^{2}(1+\left\Vert w_{1}+...+w_{r-1}\right\Vert ^{2})^{m_{r-1}%
}+\frac{\lambda_{r}^{2}}{4}\left\Vert w_{r}\right\Vert ^{2}
\]
using the inductive hypothesis with $1\geq m_{r-1}>0$. Setting $L=\min
\{C_{r-1}^{2},\frac{\lambda_{r}^{2}}{4}\}$ then%
\[
\left\Vert \varphi(w_{1}+...+w_{r})\right\Vert ^{2}\geq L((1+\left\Vert
w_{1}+...+w_{r-1}\right\Vert ^{2})^{m_{r-1}}+\left\Vert w_{r}\right\Vert
^{2}).
\]
In Lemma \ref{trivial} take $a=1+\left\Vert w_{1}+...+w_{r-1}\right\Vert ^{2}
$ and $b=\left\Vert w_{r}\right\Vert ^{2}$ then we have%
\[
\left\Vert \varphi(w_{1}+...+w_{r})\right\Vert ^{2}\geq\frac{L}{2}%
((1+\left\Vert w_{1}+...+w_{r-1}\right\Vert )^{2}+\left\Vert w_{r}\right\Vert
^{2})^{m_{r-1}}.
\]

If
\[
w_{1}+...+w_{r}\notin S_{r}
\]
then%
\[
\left\Vert u_{r}(w_{1},...,w_{r-1})\right\Vert \geq\frac{1}{2}\left\Vert
T_{r}w_{r}\right\Vert \geq\frac{\lambda_{r}}{2}\left\Vert w_{r}\right\Vert .
\]
The inductive hypothesis implies that
\[
\left\Vert \varphi(w_{1}+...+w_{r-1})\right\Vert ^{2}\geq C_{r-1}(1+\left\Vert
w_{1}+...+w_{r-1})\right\Vert ^{2})^{m_{r-1}}
\]
also since $u_{r}(w_{1},...,w_{r-1})$ is a polynomial of degree $d_{r}$
\ there exists $M_{r}$ such that
\[
\left\Vert u_{r}(w_{1},...,w_{r-1})\right\Vert \leq M_{r}(1+\left\Vert
w_{1}+...+w_{r-1}\right\Vert ^{2})^{\frac{d_{r}}{2}}.
\]
Thus
\[
\left\Vert u_{r}(w_{1},...,w_{r-1})\right\Vert \leq\frac{M_{r}}{C_{r-1}%
}\left\Vert \varphi_{r-1}(w_{1}+...+w_{r-1})\right\Vert ^{\frac{d_{r}%
}{2m_{r-1}}}.
\]
This implies that
\[
\left\Vert w_{r}\right\Vert ^{2}+C_{r-1}(1+\left\Vert w_{1}+...+w_{r-1}%
)\right\Vert ^{2})^{m_{r-1}}\leq\left(  \frac{2M_{r}}{\lambda_{r}C_{r-1}%
}\right)  ^{2}\left\Vert \varphi_{r-1}(w_{1}+...+w_{r-1})\right\Vert
^{\frac{d_{r}}{2m_{r-1}}}
\]%
\[
+\left\Vert \varphi_{r-1}(w_{1}+...+w_{r-1})\right\Vert
\]
Let $L=\min\{C_{r-1},1\}$ then
\[
\left\Vert w_{r}\right\Vert ^{2}+C_{r-1}(1+\left\Vert w_{1}+...+w_{r-1}%
)\right\Vert ^{2})^{m_{r-1}}\geq L(\left\Vert w_{r}\right\Vert ^{2}%
+C_{r-1}(1+\left\Vert w_{1}+...+w_{r-1})\right\Vert ^{2})^{m_{r-1}})
\]%
\[
\geq L_{1}(1+\left\Vert w_{1}+...+w_{r}\right\Vert ^{2})^{m_{r-1}}.
\]
Thus
\[
L_{1}(1+\left\Vert w_{1}+...+w_{r}\right\Vert ^{2})^{m_{r-1}}\leq
L_{2}\left\Vert \varphi_{r-1}(w_{1}+...+w_{r-1})\right\Vert ^{\max
\{\frac{d_{r}}{2m_{r-1}},1\}}.
\]
We have shown $w_{1}+....+w_{r}\notin S_{r}$ there exists $L_{3}>0$ such that%
\[
\left\Vert \varphi_{r}(w_{1}+...+w_{r})\right\Vert \geq L_{3}(1+\left\Vert
w_{1}+...+w_{r}\right\Vert ^{2})^{\frac{m_{r-1}}{\max\{\frac{d_{r}}{2m_{r-1}%
},1\}}}.
\]
To complete the inductive step take $C_{r}=\min\{L_{3},C_{r-1}\}$ and
$m_{r}=\min\{m_{r-1,}\frac{m_{r-1}}{\max\{\frac{d_{r}}{2m_{r-1}},1\}}\}$.
\end{proof}

In the proof of the above lemma we used

\begin{lemma}
\label{trivial}If $a\geq1$ and $b\geq0$ and if $0<m\leq1$ then $a^{m}%
+b\geq\frac{1}{2}(a+b)^{m}$.
\end{lemma}

\begin{proof}
If $b\geq1$ then $b\geq b^{m}$ thus $a^{m}+b\geq a^{m}+b^{m}$. If $b<1$ then
$a^{m}\geq b^{m}$ so $a^{m}+b\geq a^{m}\geq\frac{1}{2}(a^{m}+b^{m})$. Thus in
all cases%
\[
a^{m}+b\geq\frac{1}{2}(a^{m}+b^{m})
\]
Noting that if $x,y\geq0$ and $x+y=1$ then if $0<m\leq1$ then $x^{m}+y^{m}%
\geq1.$ This if $c,d\geq0$%
\[
\left(  \frac{c}{c+d}\right)  ^{m}+\left(  \frac{d}{c+d}\right)  ^{m}\geq1
\]
so
\[
c^{m}+d^{m}\geq(c+d)^{m}.
\]
The lemma now follows.
\end{proof}

\section{Appendix: Basics on parabolic subgroups}

Let $G$ be a connected real reductive group as in \cite{BoWa}.More detail on
the material in this appendix can be found in \cite{RRGI}. As is standard, we
will denote the Lie algebra of a Lie Group, $H$, by $\mathfrak{h}$. Let $K $
be a maximal compact subgroup of $G$ and let $\theta$ be the corresponding
Cartan involution on $\mathfrak{g}$ (and $G$). Fix a non-degenearate
$Ad(G)$--invariant symmetric bilinear form, $B$, on $\mathfrak{g}$ such that
the form
\[
\left\langle X,Y\right\rangle =-B(X,\theta Y)
\]
is positive definite.

We will take a pragmatic approach to parabolic subgroups. Let $H\in
\mathfrak{g}$ be an element such that $ad(H)$ is diagonalizable over
$\mathbb{R}$ then the direct sum of the eigenspaces of $ad(H)$ with
non-negative eidgenvalue, $\mathfrak{p}$, is a parabolic subalgebra and all
parabolic subalgebras of $\mathfrak{g}$ are obtained in this way. Let
$\mathfrak{n}$ be the direct sum of the eigenvalues of $adH$ with strictly
positive eigenvalue. Then every element of $\mathfrak{n}$ is nilpotent. Also,
set $\mathfrak{m}=\ker ad(H)$ then $\mathfrak{m}$ is reductive. Also,
\[
\mathfrak{p}=\mathfrak{m\oplus\mathfrak{n}}
\]
and $\mathfrak{n}$ is the nilpotent radical of $\mathfrak{p}$. A parabolic
subgroup of $G$ is the normalizer in $G$ of, $\mathfrak{p}$, a parabolic subalgebra.

Let
\[
\mathfrak{q}=\{X\in\mathfrak{g}|\theta X=-X\}.
\]
Fix a subspace $\mathfrak{a\subset q}$ maximal subject to the condition that
$[\mathfrak{a},\mathfrak{a}]=0$. If $X\in\mathfrak{q}$ then $ad(X)$ is
self-adjoint relative to $\left\langle ...,...\right\rangle $. Let
$\Phi(\mathfrak{g},\mathfrak{a})$ be the set of weights (i.e. simultaneous
eigenvalues) of the elements of $\mathfrak{a}$ on $\mathfrak{g}$. Then
$\Phi(\mathfrak{g},\mathfrak{a})$ is a (generally non-reduced) root system.
Also, all such maximal abelian subspaces of $\mathfrak{q}$ are conjugate
relative to $Ad(K)$. If $\alpha\in\Phi(\mathfrak{g},\mathfrak{a})$ then set
$\mathfrak{g}_{\alpha}=\{X\in\mathfrak{g}|[h,X]=\alpha(h)X,h\in\mathfrak{a\}}%
$. Let $\Phi^{+}(\mathfrak{g},\mathfrak{a})$ be a system of positive roots for
$\Phi(\mathfrak{g},\mathfrak{a})$ then if $\mathfrak{n}_{o}\mathfrak{=}%
\sum_{\alpha\in\Phi^{+}(\mathfrak{g},\mathfrak{a})}\mathfrak{g}_{\alpha}$ and
$\mathfrak{m}_{o}\mathfrak{=\mathfrak{c}}_{\mathfrak{g}}(\mathfrak{a}%
)=\{X\in\mathfrak{g}|[\mathfrak{a},X]=0\}$ then $\mathfrak{p}_{o}%
\mathfrak{=\mathfrak{m}}_{o}\mathfrak{\mathfrak{\oplus\mathfrak{n}}}_{o} $ is
a minimal parabolic subalgebra of $\mathfrak{g}$ and its normalizer in $G$,
$P_{o}$, is a minimal parabolic subgroup of $G$.. All minimal parabolic
subalgebras of $\mathfrak{g}$ are conjugate under $K$. Thus if $P$ is a
minimal parabolic subgroup of $G$ then $KP=G$.

Fix, $P_{o}$, a minimal parabolic subgroup of $G$ Then we may assume that it
has been constructed as above using a maximal abelian subspace, $\mathfrak{a}
$, of $\mathfrak{q}$. Let $\Delta(\mathfrak{g},\mathfrak{a})$ be the set of
simple roots in $\Phi^{+}(\mathfrak{g},\mathfrak{a})$ and let $H_{o}%
\in\mathfrak{a}$ be such that $\alpha(H_{o})=1$ if $\alpha\in\Delta
(\mathfrak{g},\mathfrak{a})$ and such that the center of $\mathfrak{g}$ is
contained in $\ker ad(H_{o}).$ Let $P$ be a parabolic subgroup of
$\mathfrak{g}$ containing $P_{o}$. We may assume that the element $H$ that was
used above in the definition of $\mathfrak{p}$ is in $\mathfrak{a}$. This
implies that if $\mathfrak{p}=\mathfrak{m}_{P}\mathfrak{\oplus n}_{P}$ as
above then $\mathfrak{n}_{P}\subset\mathfrak{n}_{o}$ and $\mathfrak{m}%
_{P}\supset\mathfrak{m}_{o}$. Thus $ad(H_{o})$ normalizes $\mathfrak{m}%
_{P}\mathfrak{\ } $and $\mathfrak{n}_{P}$. Define $T:\mathfrak{g}%
\rightarrow\mathfrak{g}$ by $T_{|\theta\mathfrak{n}_{P}}=adH_{o|\theta
\mathfrak{n}_{P}},T_{|m_{P}}=0$ and $T_{|\mathfrak{n}_{P}}=adH_{o|\mathfrak{n}%
_{P}}$ then it is easily checked that $T$ defines a derivation on
$\mathfrak{g}$ and since $\ker T$ contains the center of $\mathfrak{g}$ it is
given by $adH^{\prime}$ for $H^{\prime}\in\mathfrak{a.}$ Thus if $P$ is a
parabolic subgroup of $G$ then there exists a grade of $\mathfrak{g}$,
$\mathfrak{g=\oplus}_{j\in\mathbb{Z}}\mathfrak{g}_{j}$ such that
$\mathfrak{m=\mathfrak{g}}_{0}$ and $\mathfrak{n=\oplus}_{j\geq1}%
\mathfrak{g}_{j}$. One can also check that $\mathfrak{g}_{1}\neq0$ with the
choice of $H$ above.

If $P\subset G$ is a subgroup that is its own normalizer and contains a
minimal parabolic subgroup of $G$ then $P$ is a parabolic subgroup of $G$. Fix
a minimal parabolic subgroup, $P_{o}$, with $\mathfrak{p}_{o}=\mathfrak{m}%
_{o}\oplus\mathfrak{n}_{o}$ as above. Then $\mathfrak{m}_{o}=\theta
(\mathfrak{p}_{o})\cap\mathfrak{p}_{o}$. Furthermore, $\mathfrak{m}%
_{o}=\mathfrak{a}_{o}\oplus\mathfrak{m}_{o}\cap\mathfrak{k}$. If $P$ is a
parabolic subgroup containing $P_{o}$ then $\mathfrak{m}_{p}=\theta\left(
\mathfrak{p}\right)  \cap\mathfrak{p}$ and $M_{P}=\theta(P)\cap P$. Let
$C_{P}$ be the center of $M_{P}$ and let $C_{P}^{o}$ be the identity component
of $C_{P}$ and $A_{P}=A_{o}\cap C_{P}^{o}$. Setting ${}^{o}M_{P}=[M_{P}%
,M_{P}]K\cap M_{P}$ we have%
\[
G=N_{P}{}^{o}M_{P}A_{P}K.
\]
If $g\in G$ then define $a_{P}(g)$ by $g\in N_{P}{}^{o}M_{P}a_{P}(g)K$ then
$a_{P}:G\rightarrow A_{P}$ is of class $C^{\infty\text{. }}$ The
decomposition
\[
P=N_{P}{}^{o}M_{P}A_{P}
\]
is called the Langlands decomposition of $P$.

\end{document}